\documentclass[12pt]{article}
\usepackage{amsmath,amssymb}
\usepackage{graphics}
\usepackage{epic}
\usepackage{epsfig}
\usepackage{subfigure}

\newtheorem{thm}{Theorem}[section]
\newtheorem{lem}[thm]{Lemma}

\newtheorem{prop}[thm]{Proposition}
\newtheorem{rema}[thm]{Remark}

\newenvironment{proof}[1][]%
 {\ \\ {\bf Proof #1. }}%
 {\hfill\mbox{\rule{2 true mm}{3 true mm}}}%\vskip 2 ex\noindent}

\numberwithin{equation}{section}
\def\theequation{\thesection.\arabic{equation}}
\def\be{\begin{eqnarray}}
\def\ee{\end{eqnarray}}
\def\numero{\refstepcounter{equation} (\theequation)}

\let\text=\textstyle

\newcommand{\N}{{\mathbb N}}
\newcommand{\R}{{\mathbb R}}

\newcommand{\loi}{{\cal L}}

\newcommand{\ds}{\displaystyle}
\newcommand{\intot}{\displaystyle \int _0^t }

\newcommand{\E}{{E}}

\newcommand{\intrd}{\ds{\int_{\rit^d}}}
\newcommand{\rit}{\mathbb{R}}

\newcommand{\dit}{\mathbb{D}}

\def\be{\begin{eqnarray}}
\def\ee{\end{eqnarray}}
\def\ben{\begin{eqnarray*}}
\def\een{\end{eqnarray*}}
\def\bei{\begin{itemize}}
\def\eei{\end{itemize}}
\def\me{\medskip \noindent}
\def\bi{\bigskip \noindent}

\title{\bf L\'evy Flights in Evolutionary Ecology}

\author{Benjamin Jourdain$^{1}$, Sylvie
M\'el\'eard$^{2}$, Wojbor A. Woyczynski$^{3}$}

\date{\today}

\begin{document}

\maketitle

\begin{center}
\makeatletter\renewcommand{\@makefnmark}{\mbox{$^{\@thefnmark}$}}\makeatother
\footnotetext[2]{Universit\'e Paris Est, CERMICS, 6 et 8 avenue Blaise Pascal, Cit\'e Descartes, Champs-sur-Marne, 77455 Marne-la-Vall\'ee Cedex 2, France, email:
jourdain@cermics.enpc.fr}
\footnotetext[2]{CMAP UMR 7641, Ecole Polytechnique, Route de Saclay, France, email:
sylvie.meleard@polytechnique.edu}
\footnotetext[3]{Department of Statistics, and Center for Stochastic and Chaotic Processes in Science and Technology, Case Western Reserve University, Cleveland, OH 44122, U.S.A., email: waw@case.edu}
\makeatletter\renewcommand{\@makefnmark}{}\makeatother
\makeatletter\renewcommand{\@makefnmark}{\mbox{$^{\@thefnmark}$}}\makeatother
\end{center}

% \enlargethispage*{1000pt}

\begin{abstract}
We are interested in modeling Darwinian evolution 
resulting from the interplay of phenotypic variation
and natural selection through ecological interactions.
The
population is modeled as a stochastic point process whose generator
captures the probabilistic dynamics over continuous time of birth,
mutation, and death, as influenced by each individual's trait
values, and interactions between individuals. An offspring usually
inherits the trait values of her progenitor, except when a random 
mutation causes the offspring to take an instantaneous
mutation step at birth to new trait values.  In the case we are interested in, the probability distribution of mutations has a  heavy tail and belongs to the  domain of  attraction of a stable law.  We investigate the large-population limit
with allometric demographies: larger populations made up of smaller individuals which reproduce and die faster, as  is typical for  micro-organisms.  We show that depending on the allometry coefficient the limit behavior of the population process can be approximated by nonlinear L\'evy flights  of different nature: either deterministic, in the form of nonlocal fractional reaction-diffusion equations, or  stochastic, as nonlinear  super-processes with the  underlying reaction and  a fractional diffusion operator. These approximation results demonstrate  the existence of such nontrivial fractional objects; their uniqueness is also proved. 
\end{abstract}

\bigskip


% \pagebreak

%%%%%%%%%%%%%%%%%%%%%%%%%%%%%%%%%%%%%%%%%%%%%%%%%%%%%%%%%%%%%

\section{Introduction}
\label{sec:intro}

In this paper, we are interested in modeling  the dynamics of
populations as driven by the interplay of phenotypic
variation and natural selection operating through ecological
interactions, i.e.,  Darwinian evolution.   The population is modeled as a stochastic Markov
point process whose generator captures the probabilistic dynamics
over continuous time of birth, mutation and death, as influenced
by each individual's trait values and interactions between
individuals. The adaptive nature of a trait implies that an
offspring usually inherits the trait values of her progenitor,
except when a mutation occurs. In this case, the offspring makes
an instantaneous mutation step at birth to new trait values.
This microscopic point of view has been heuristically
introduced in Bolker-Pacala~\cite{BP97} and
Dieckmann-Law~\cite{DL00}. It has been  rigorously developed first
in Fournier-M\'el\'eard \cite{FM04} for spatial seed models and by
Champagnat-Ferri\`ere-M\'el\'eard \cite{CFM06}, \cite{CFM08} for
phenotypic trait structured populations when the mutation kernel behaves essentially as a Gaussian law.
The  aim in this work is to study the case  where a mutant
individual can be significantely different from his ancestor.
More precisely, the mutation kernel will be assumed to have  a 
heavy tail and to belong to the domain of attraction  of a
stable law.
If the traits describe a spatial dispersion, for instance for seeds,  we thus assume
  that seed offsprings instantaneously  jump far from the mother seed because of
  availability of resources or wind. In the case of  phenotypic traits, the heavy tail mutation assumption says  that if a mutant offspring is too close to the mother trait, then it is so deleterious that it cannot be
  observed. A trait which quantifies the aggressivity level is an example of such a situation.

\bigskip \noindent  In the context of ecology several authors have considered the fractional reaction-diffusion model and we would like to acknowledge some of their contributions here. In particular,  Baeumer, Kovacs, and Meerschaert \cite{BKM07} considered fractional repro\-duction-dispersal equations and heavy tail dispersal kernels. Their paper also contains an exhaustive review of the literature on the population spreads with extreme patterns of dispersal and reproduction.  Another population dynamics model introduced by Gurney and Nisbet \cite{GN75} led to strongly nonlinear partial differential equation of the porous media type; we have studied a   probabilistic fractional framework for  related stochastic and partial differential  equations in \cite{JMW08}.

\bigskip \noindent Each individual is characterized by a real-valued  trait $x$ describing   a phenotypic,  or a spatial parameter. The birth and death rates of this individual  depend on its trait $x$ and also on the environment through the influence of the others individuals alive. In case the offspring
produced by  the individual with trait $x$ carries a mutated trait, then the difference between this mutated trait and $x$ is distributed according to $M(x,dh)$ 
such that  
  \be 
  \label{mutation}
  \int_{|h|\geq y}  M(x,dh) \sim_{y\to \infty} {C(x)\over
  y^\alpha},
  \ee
  with $\alpha\in (0,2)$ not depending on $x$. The inclusion of the dependence of the mutation rate on the trait itself, as expressed in the dependence of $M$, and $C$, on $x$, is an essential part of our model even though it considerably complicates the mathematical framework of this paper.

\bigskip \noindent
 An example for $M(x,dh)$ could be  the Pareto law with density
  ${\alpha\over 2} {\bf 1}_{\{|h|\geq 1\}} {/
  |h|^{1+\alpha}}$, but
 a mutation law  equal to the
  Pareto law (up to a constant) outside a given interval of the form $[-a, +a]$,
   and constant inside, is also   covered by our model. The latter example corresponds to a  
   distribution of mutant traits which is uniform in a small neighborhood around the mother's trait, but 
   decreases, with heavy tail  for more distant traits.

\bigskip \noindent
We investigate  the large population limits, $K\to \infty, $ 
with allometric demographies: larger populations made up of smaller individuals which reproduce and die faster. This leads to systems in which organisms have short lives and reproduce fast while their colonies or populations grow or decline on a slow time scale. Typically, these assumptions are appropriate for microorganisms such as bacteria or plankton.  The allometric effect will be modeled by a dominant birth and death rate of order $K^\eta$, $\eta>0$. 

\bigskip\noindent
 In \cite{CFM08}, such asymptotic approximations have been studied in the case of small mutations with a Gaussian, thin tail  distribution. In such a situation, for  the allometric exponent $\eta<1$, and large enough   population,    the individual population stochastic process
is approximated by the solution of a classical deterministic  
nonlinear reaction-diffusion equation.  In our present heavy tail setting  we will prove that in a similar limit, the
population process is approximated by the solution of a
nonlocal nonlinear partial differential equation driven by a
fractional Laplacian and involving a reaction term. 
Note that, as a byproduct,  this result provides a proof  of  the  existence of weak solutions of such nonclassical equations. Separately, employing purely  analytical techniques, we will also 
prove the uniqueness of their solutions.  The case of the allometric exponent $\eta=1$ is significantly different as stochasticity remains present in the limit due to the demographic acceleration. The population process is then approximated by a nonlinear super-process with underlying reaction and fractional diffusion. 

\bigskip\noindent
 Our work  provides a rigorous derivation of macroscopic models  from microscopic dynamics (hydrodynamic limit) for  a large class of nonlinear equations involving    L\'evy flights which naturally appear in
 evolutionary ecology and population dynamics. Our use of probabilistic tools of interacting particle systems is essential  and provides the information  about  the  scales (between mutation amplitude and population size) at which 
such models are justified.

 \bigskip\noindent
 It should be mentioned at this point that fractional reaction-diffusion equations have been suggested and studied as models for several physical phenomena. Thus, Del-Castillo-Negrete, Carreras, and Lynch \cite{DCL03} investigated front dynamics in reaction-diffusion systems with L\'evy flights, and Henry, Langlands and Wearne \cite{HLW05}, studied Turing pattern
formation in fractional activator-inhibitor systems. Also, Saxena, Mathai, and Haubold \cite{SMH06}  found explicit solutions of the fractional reaction-diffusions equations in terms of the Mittag-Leffler functions which are suitable for numerical computations; they also considered the situation where the time derivative is also replaced by a fractional derivative of order less than one. Another numerical method for finding solutions of such equations can be found in Baeumer, Kovacs, and Meerschaert \cite{BKM08}.

\bigskip \noindent
This paper starts (Section 2) with the 
description of  the reproduction and death mechanisms for the individuals of the  population we are interested in. The main convergence results based on a large population
limit are stated. Thus Theorem~\ref{readif} shows that an allometric effect of order $K^\eta$, with $\eta\in (0,1)$,  leads to a deterministic, nonlinear
integro-dif\-fe\-ren\-tial equation driven by a nonlocal fractional Laplacian operator  independent of $\eta$, while  in
Theorem~\ref{readifstoch}, we show that an allometric effect of order $K$ ($\eta=1$) yields,  in the large population limit, a stochastic measure-valued
process depending on the acceleration rate of the birth-and-death
process.  Thus (demographic) stochasticity  appears as the allometric exponent takes on the value $\eta=1$.  Also, 
Proposition \ref {restlapfrac} establishes 
that  the mutation kernel conveniently renormalized behaves approximately as a jump kernel with the heavy tail jump measure  ${\sigma(x)dz/|z|^{1+\alpha}}$.  
Section 4 contains technical lemmas needed in the proof of the main results of Section 3.  In particular we clarify a key technical  point that was eluded in \cite{FM04}, \cite{CFM08}, allowing to deduce the tightness of the measure valued population process for the weak topology from the tightness for the vague topology (see Lemma \ref{controle-masse} and  Remark \ref{rem-contmass} and Step 2 of the proof of Theorem \ref{readif}). Section 5 contains proofs of the main theorems stated in Section 3  using the  measure-valued martingale properties of the
population  process.    %%%%%%%%%%%%%%%%%%%%%%%%%%%%%%%%%%%%%%%%%%%%%%%%%%%%%%%%%%%%%%

\section{Population point process}
\label{sec:ppp}

As in \cite{CFM08}, the evolving population is modelled by a stochastic system of interacting
individuals, where each individual is characterized by a
phenotypic trait. This trait is described quantitatively by a real number.
We assume that the parameter $K$ scales the initial number of
individuals. To observe  a nontrivial limit behavior of the system  as $K$ grows to
infinity it is necessary to attach to  each individual  the weight ${1\over
K}$. Thus our system evolves in the subset ${\cal M}_K$ of  the set $M_F$ of finite non-negative measures on
$\mathbb{R}$    consisting
of all finite point measures with weight ${1\over K}$:
\begin{equation*}
{\cal M}_K = \left\{ {1\over K}\sum_{i=1}^n \delta_{x_i} , \; n \geq 0,
  x_1,...,x_n \in \mathbb{R} \right\}.
\end{equation*}
Here and below, $\delta_x$ denotes the Dirac mass at $x$. For any
$m\in M_F$, any measurable function $f$ on $\mathbb{R}$, we set
$\left< m, f \right> = \int_{\mathbb{R}} f dm$.
Following \cite{CFM06}, we describe the population by
\begin{equation}
\label{pop}
\nu^K_t = {1\over K} \sum_{i=1}^{I^K_t} \delta_{X^i_t},
\end{equation}
with $I^K_t \in {\mathbb{N}}$ denoting  the number of individuals
alive at time $t$, and $X^1_t,...,X^{I^K_t}_t$ describing the
individuals' traits (in $\mathbb{R}$).

\me
The population measure-valued process $\nu^K$ evolves as a birth
and death process with mutation and selection. More precisely, an
individual can give birth or die.  The death can be natural, or can be due to
the competition pressure exerted by other individuals (for instance, 
by sharing food).
At birth, the offspring can inherit the   trait of its parent,  or can
mutate to another trait with some positive probability.  We assume that
the population is of order $K$ and that for each individual, 
  birth and death occur at the rate of order  $K^{\eta}$,
for some $0<\eta \leq 1$,
while preserving the demographic balance. More precisely, the main assumptions on
the birth and death dynamics are summarized below.

\begin{itemize}
\item {\bf Scaling Assumptions:} For a population $\nu={1\over K}\sum_{i=1}^{I}\delta_{x^i}$,
and a trait $x\in \mathbb{R}$,
  the birth and   death rates are scaled with the system's size according
to the following rules:
 \be
b_K(x,\nu)&=&K^{\eta}r(x)+b(x, {1\over K} \sum_{i=1}^{I}V(x-x^i))
= K^{\eta}r(x)+b(x, V*  \nu(x)),\nonumber\\
d_K(x,\nu)&=&K^{\eta}r(x)+ d(x, {1\over K}\sum_{i=1}^{I}U(x-x^i))
= K^{\eta}r(x)+d(x, U*\nu(x)),\nonumber\\\label{taux}    
   \ee
where $b$ and $d$ are continuous functions on $\mathbb{R}^2$, and $*$ denotes the convolution operation.
The allometric effect (smaller individuals reproduce and die faster) is parametrized by  the exponent $\eta$ and  a trait-dependent function $r$  which is assumed to be positive and bounded on $\mathbb{R}$.

  \item{\bf Assumptions (H1):}     The interaction kernels $V$, 
  and $U$, affecting,  respectively,  
  reproduction and mortality rates, are continuous functions on $\mathbb{R}$. In addition, 
  there exist constants $\bar{r}$, $\bar{b}$,
$\bar{d}$, $\bar{U}$, and $\bar{V}$,  such that, 
for $x,z\in
\mathbb{R}$,
\be \label{boundedness} 0&\leq& r(x)\leq \bar{r}\ ;\ 
0\leq b(x,z)\leq \bar{b},\\
\label{linearity}  0&\leq&d(x,z)\leq
\bar{d}\, (1+|z|), \\
0&\leq&U(x)\leq \bar{U}\ ;\ 
0\leq  V(x)\leq \bar{V}.
\ee
Note that the death rate $d$ is not assumed to be bounded but its growth in variable   $z$ is at most linear. This dependence models a possible competition between individuals (e.g.,  for shared resources), increasing their death rate. 
  Assumptions (H1) ensure that there exists a constant
$\bar{C}>0$, such that the total event rate for a population
counting measure $\nu=\frac{1}{K}\sum_{i=1}^{I}\delta_{x^i}$, obtained as the
sum of all event rates, is bounded by $\ \bar{C}I(1+I)\ $, where $\, I$ is the population size.

\item  {\bf Assumption (H2):}  When an individual with trait $x$ gives birth, it can
produce a mutant offspring with  probability $p(x)$. Otherwise,
with probability $1-p(x)$, the
offspring carries the same trait $x$ as its ancestor.
If a mutation occurs, the mutated offspring instantly acquires a
new trait $x+h$, where $h$ is selected randomly according to the mutation step
measure $M_K(x,dh)$.
%We  assume that there is a probability measure $\bar{M}_{K}$ on the real line such that, for each $x\in\rit$, $M_K(x,dh)$ is absolutely continuous with respect to $\bar{M}_{K}$. 
 We only consider  mutations which have heavy tail distributions. More precisely, we assume that 
the probability measure $M_K(x,dh)$ is
the law of the random variable
 \ben {X(x)\over K^{\eta\over
\alpha}},
\een
 where $\ X(x)$ is a symmetric random variable such
that, for some $\alpha\in (0,2)$, and a bounded $\sigma:\rit\to\rit_+$, 
 \be
  \label{heavytail} \lim_{u\to+\infty}\sup_{x\in\rit}\left|u^\alpha\mathbb{P}(|X(x)|\geq u)-\frac{2\sigma(x)}{\alpha}\right|=0. 
\ee

\end{itemize}

\bigskip \noindent
As an immediate corollary of Assumption (H2) we have 
$$\int_{|h|\geq y}  M_{K}(x,  dh) \sim_{y K^{\frac{\eta}{\alpha}}\to \infty} {2\, \sigma(x)\over
\alpha\,K^{\eta}\,  y^\alpha},$$
and, for  $\alpha\in (1,2)$,   also  
$$\int_{\rit} h M_K(x,dh)= 0.$$

\bigskip

{\it Example 2.1.}  One can choose the random variable  $X(x)$ with
Pareto's law independent  of  $x$, that is, with the density 
\ben
{\alpha\over 2} {\bf 1}_{\{|h|\geq 1\}} {1\over
  |h|^{1+\alpha}}.
  \een
  
  \bigskip
  {\it Example 2.2.} 
Another possibility is to take $X(x)$ 
  with the 
  Pareto law (up to multiplicative a constant) outside a given interval of the form $[-a, +a]$, 
   and constant inside. This choice    corresponds to a  
   distribution of mutant traits which is uniform in a small neighborhood around the mother's trait, but 
   decreases, with heavy tail.  for more distant traits.

% We assume that there exists  a constant $C$
%   and a probability density function $\bar{M}_K$ on $\rit$  such
%   that for all $x, h\in \mathbb{R}$,
%   \be
%   \label{mutation}
%   M_K(x,h)\leq C\bar{M}_K(h).
% \ee

  \bi
  We refer to Fournier-M\'el\'eard \cite{FM04} or Champagnat-Ferri\`ere-M\'el\'eard \cite{CFM06}
 for a pathwise construction of   a point measure-valued Markov process $(\nu^K_t)_{t\geq 0}$
satisfying  Assumptions (H1). 
The  infinitesimal generator of its  Markovian dynamics  
is given, for each finite point measure $\nu$,  by the expression
\begin{align} \label{generator}
%   \label{eq:def-LK}
L^K\phi(\nu)&=K\int_{\rit} (1-p(x))\ (K^{\eta}
r(x)+b(x,V*\nu(x)))(\phi(\nu+{1\over
  K}\delta_x)-\phi(\nu)) \nu(dx) \nonumber
\\ & \hskip -1cm +K\int_{\rit} p(x)\ (K^{\eta}
r(x)+b(x,V*\nu(x)))\int_{\rit}(\phi(\nu+{1\over
  K}\delta_{x+h})-\phi(\nu))M_K(x,dh) \nu(dx) \nonumber
\\ & +K\int_{\rit}
(K^{\eta} r(x)+d(x,U*\nu(x)))(\phi(\nu-{1\over
  K}\delta_x)-\phi(\nu)) \nu(dx).
\end{align}
 The first term of~(\ref{generator}) captures the effect  of births without mutation, the second term that of
births with mutation, and the last term that of deaths.

\bigskip\noindent
 If
$\E(\langle\nu^K_0,\mathbf{1}\rangle^2)<+\infty$, then 
for any
  $T<\infty$,   $ \E (\sup_{t\in[0,T]} \left< \nu^K_t,1\right>^2 ) <\infty$ (see Lemma \ref{mep}). 
Thus, for any measurable functions $\phi$ on $M_F$ such that
$|\phi(\nu)|+|L^K\phi(\nu)|\leq C(1+\langle\nu,1\rangle^2)$, the
process
\begin{equation}
\label{eq:mart-gal}
\phi(\nu^K_t)-\phi(\nu^K_0)-\int_0^tL^K\phi(\nu^K_s)ds
\end{equation}
is a martingale. In particular, in view of (\ref{generator}), for each measurable bounded
function $f$,  
\begin{align}
& M^{K,f}_t= \langle \nu^K_t,f\rangle -
\langle \nu^K_0,f\rangle \notag \\ &
-\int_0^t\!\!\int_{\rit}(b(x,V* \nu^K_s(x))-d(x,U*\nu^K_s(x)))f(x)
\nu^K_s(dx)ds \label{mart1} \\ & -\int_0^t\!\!\int_{\rit}
p(x)\ (K^{\eta}r(x)+b(x,V*\nu^K_s(x)))
\bigg(\!\int_{\rit}f(x+h)M_K(x,dh)-f(x)\bigg)\nu^K_s(dx)ds, \notag
\end{align}
is a square integrable martingale with quadratic variation
\begin{multline}
\label{quadra1}
\langle M^{K,f}\rangle_t={1\over K}\bigg\{
\int_0^t\!\!\int_{\rit}(2K^{\eta}r(x)+ b(x,V*\nu^K_s(x))+d(x,U*
\nu^K_s(x)))f^2(x)\nu^K_s(dx)ds \\
+\int_0^t\!\!\int_{\rit} p(x)\ (K^{\eta}r(x)+b(x,V*\nu^K_s(x)))
\bigg(\int_{\rit}f^2(x+h)M_K(x,dh)- f^2(x)\bigg)\nu^K_s(dx)ds
\bigg\}.
\end{multline}

\section{Scaling limits of population point processes}

\bi Our aim is now to make   the population size's scaling $K$ tend to infinity, in a scale accelerating the births and deaths, and making
the mutation steps smaller and smaller.
We begin by explaining  why a fractional Laplacian term appears   in the limit,  $K\to\infty$, of the last term of the r.h.s. of \eqref{mart1}.

\begin{prop}\label{restlapfrac}
 Under {\bf (H2)},
 
 (i)   if $\alpha\in (1,2)$, then 
for any $f\in C^2_b(\rit)$, $K^\eta\int_\rit (f(x+h)-f(x))M_K(x,dh)\ $ converges to $\ \sigma(x)\int_{\rit}(f(x+z)-f(x)-f'(x)z{\bf 1}_{\{|z|\leq 1\}})\frac{dz}{|z|^{1+\alpha}}$ uniformly for $x\in\rit$ as $K\to\infty$.

 (ii) if $\alpha\in (0,1]$, then  the same result holds for  $f\in C^2_b(\rit)$ compactly supported. 
\end{prop}

\begin{proof}
 By Fubini's theorem and since $f(x+h)-f(x)=\int_0^{h}f'(x+z)dz$,
$$\int_\rit {\bf 1}_{\{h\geq 0\}}(f(x+h)-f(x))M_K(x,dh)=\int_0^\infty f'(x+z){\mathbb P}(X(x)\geq K^{\frac{\eta}{\alpha}}z)dz.$$
Treating in the same way the integral for $h<0$ and using the symmetry of $X(x)$, one deduces that
$$\int_\rit (f(x+h)-f(x))M_K(x,dh)=\int_0^\infty\frac{f'(x+z)-f'(x-z)}{2}{\mathbb P}(|X(x)|\geq K^{\frac{\eta}{\alpha}}z)dz.$$
By integration by parts,
\begin{align*}
 \int_0^{+\infty}\left(f'(x+z)-f'(x-z)\right)\frac{dz}{z^\alpha}
&=\alpha\int_0^{+\infty} (f(x+z)+f(x-z)-2f(x))\frac{dz}{z^{1+\alpha}}
\\&=\alpha\int_\rit (f(x+z)-f(x)-f'(x)z{\bf 1}_{\{|z|\leq 1\}})\frac{dz}{|z|^{1+\alpha}}.
\end{align*}
Combining both equalities one deduces that
\begin{align*}
 R(x)&\stackrel{\rm def}{=}K^\eta\int_\rit (f(x+h)-f(x))M_K(x,dh)-\sigma(x)\int_{\rit}(f(x+z)-f(x)-f'(x)z{\bf 1}_{\{|z|\leq 1\}})\frac{dz}{|z|^{1+\alpha}}\\
&=\int_0^{+\infty}\frac{f'(x+z)-f'(x-z)}{2}\left(K^\eta z^\alpha{\mathbb P}(|X(x)|\geq K^{\frac{\eta}{\alpha}}z)-\frac{2\sigma(x)}{\alpha}\right)\frac{dz}{z^{\alpha}}.
\end{align*}
Therefore 
\begin{align}
 |R(x)|&\leq \|f''\|_\infty\sup_{u>0,x\in\rit}\Big|u^\alpha\,{\mathbb P}(|X(x)|\geq u)-\frac{2\sigma(x)}{\alpha}\Big|\int_0^{K^{-\frac{\eta}{2\alpha}}}z^{1-\alpha}dz\notag\\
&+\|f''\|_\infty\sup_{u>K^{\frac{\eta}{2\alpha}},x\in\rit}\Big|u^\alpha{\mathbb P}(|X(x)|\geq u)-\frac{2\sigma(x)}{\alpha}\Big|\int_{K^{-\frac{\eta}{2\alpha}}}^1z^{1-\alpha}dz\notag\\
&+\sup_{u>K^{\frac{\eta}{\alpha}},x\in\rit}\Big|u^\alpha{\mathbb P}(|X(x)|\geq u)-\frac{2\sigma(x)}{\alpha}\Big|\int_1^{+\infty}|f'(x+z)-f'(x-z)|\frac{dz}{2z^{\alpha}}.\label{contrest}
\end{align}
By {\bf (H2)} $\sup_{u>K^{\frac{\eta}{2\alpha}},x\in\rit}\Big|u^\alpha{\mathbb P}(|X(x)|\geq u)-\frac{2\sigma(x)}{\alpha}\Big|$ tends to $0$ as $K\to\infty$. Moreover for $v>0$,
$$\sup_{u>0,x\in\rit}\Big|u^\alpha{\mathbb P}(|X(x)|\geq u)-\frac{2\sigma(x)}{\alpha}\Big|\leq v^\alpha+\frac{2}{\alpha}\sup_{x\in\rit}\sigma(x)+\sup_{u>v,x\in\rit}\Big|u^\alpha{\mathbb P}(|X(x)|\geq u)-\frac{2\sigma(x)}{\alpha}\Big|$$
with the right-hand-side finite when $v$ is large enough.\\
Last, $\int_1^{+\infty}|f'(x+z)-f'(x-z)|\frac{dz}{2z^{\alpha}}$ is smaller than $\|f'\|_\infty\int_1^{+\infty}\frac{dz}{z^{\alpha}}$ if $\alpha>1$. If $\alpha\in (0,1]$, it is smaller  than $\|f'\|_\infty\int_1^{1+2C}\frac{dz}{z^{\alpha}}$ when the compactly supported function $f(y)$ vanishes for $|y|\geq C$.
Hence \eqref{contrest} implies the desired uniform convergence as $K\to\infty$.

\end{proof}

%%%%%%%%%%%%%%%%%%%%%%%%%%%%%

\bigskip 
  In what follows, we will denote by $\tilde\sigma$ the product function  defined by $$\tilde\sigma(x){=}p(x)r(x)\sigma(x)$$ and by $\hat{\sigma}$ the function $$\hat{\sigma}(x)=\tilde{\sigma}^{1/\alpha}(x)=\left(p(x)r(x)\sigma(x)\right)^{1/\alpha}.$$

\bigskip
  The nature of the scaling limit of the population point processes $\nu^K$ strongly depends on the value of the allometric exponent $\eta$. For $0<\eta<1$, the limit is deterministic and our convergence result is described in the following  theorem:

\begin{thm}
\label{readif}
%\begin{description}
(i)  Suppose that Assumptions (H1) and (H2) are satisfied, $0<\eta<1$ and the product function $\tilde\sigma$ is continuous.
  Additionally, assume   that, as $K\to \infty$,  the initial conditions $\nu^K_0$ converge in law,
  and for the weak topology on $M_F$, to a finite
  deterministic measure $\xi_0$, and that
  \begin{equation}
    \label{eq:X3}
    \sup_K \E(\langle \nu^K_0,1\rangle^3)<+\infty.
  \end{equation}
  Then, for each $T>0$, the laws $ (Q^K) $ of the processes  $ (\nu^K) $ in  $\dit([0,T],M_F)$ (with $M_F$ endowed with the weak convergence topology) are tight. Moreover,     the weak limit  of each of their convergent subsequences gives  full weight to the process $(\xi_t)_{t\geq 0} \in C([0,T],M_F)$
  satisfying the following condition: for each function $f\in C^2_b(\rit)$,
  \begin{align}
    \label{readif1}
    \langle\xi_t,f\rangle &
    =\langle \xi_0,f\rangle+ \int_0^t\int_{\rit}
    (b(x,V*\xi_s(x))-d(x,U*\xi_s(x)))f(x)\xi_s(dx)ds \notag \\ &
    +\int_0^t\int_{\rit}\tilde\sigma(x)\bigg(\int_\mathbb{R}
(f(x+h) -
    f(x) - f'(x)h{\bf 1}_{\{|h|\leq 1\}}) {dh\over |h|^{1+\alpha}} \bigg)\xi_s(dx)ds.
  \end{align}
\smallskip

(ii)  If we assume additionally that $b(x,z)$, and $d(x,z)$, are Lipschitz continuous in $z$, uniformly for $x\in\rit$, and that $\hat{\sigma}$ is Lipschitz continuous, then \eqref{readif1} has at most  one solution such that $\sup_{t\in[0,T]}\,\langle \xi_t,1\rangle<+\infty$, and, as $K\to\infty$, 
 the processes $(\nu^K)$ converge weakly to this unique solution. 
\smallskip

(iii) Finally, if $\hat{\sigma}\in C^3_b$, i.e., it is bounded together with its derivatives of  order  $\le$ 3,  and  the product function $\tilde{\sigma}(x)>0$, for all $  x\in\rit,$  then,  for each $t>0$, the measure $\xi_t$ has
a density
  with respect to the Lebesgue measure.
%\end{description}
\end{thm}

%\begin{rema}

\bigskip

The following remarks are immediate corollaries of   the above Theorem.
 \medskip

 \noindent {\it Remark 3.1.}   The limit~(\ref{readif1}) does not depend on 
$\eta\in(0,1)$. As will appear in the proof, this
is implied by  the fact that the growth rate $b_K-d_K$ does not depend on $\eta$, and that the mutation kernel $M_K(x,z)$
compensates exactly the dispersion in the trait space induced by the
acceleration of the births with mutations.

\medskip
\noindent  {\it Remark 3.2.}  In the case  considered in Theorem \ref{readif}  {\it (iii)}, Eq.~(\ref {readif1}) may be written in the form
\begin{equation}
  \label{reacdiff}
  \partial_t
  \xi_t(x)=\bigg(b(x,V*\xi_t(x))-d(x,U*\xi_t(x))\bigg)\xi_t(x) +
   D^{\alpha}(\tilde{\sigma}\, \xi_t)(x),
\end{equation}
where 
$$
D^\alpha f(x)=\int_\mathbb{R}
(f(x+h) -
    f(x) - f'(x)h{\bf 1}_{\{|h|\leq 1\}}) {dh\over |h|^{1+\alpha}},
    $$
    denotes the fractional Laplacian of order
$\alpha$. Note that, by the change of variable $z=h/\hat{\sigma}(x)$,
\begin{equation}
\tilde{\sigma}(x)D^\alpha f(x)=\int_\mathbb{R}
(f(x+\hat{\sigma}(x)z) -
    f(x) - f'(x)\hat{\sigma}(x)z{\bf 1}_{\{|z|\leq 1\}}) {dz\over |z|^{1+\alpha}}.
\label{chgtvar}\end{equation}

\medskip
 \noindent  {\it Remark 3.3.}  Theorem \ref{readif} also proves the existence
of a weak solution for (\ref{reacdiff}).
Equation (\ref{reacdiff}) generalizes the Fisher
reaction-diffusion equation known from classical population genetics
(see e.g.~\cite{Bu00}) with a fractional Laplacian term replacing the classical Laplacian. It
  justifies 
  the L\'evy flight modeling in ecology as approximations for models
 in which  mutations have  heavy tail. Such models were judged justified by real-life data  in many areas of physics and economics, and in some popular literature devoted to the 2008 global  financial 
crisis. 
 \bigskip
 
%\end{rema}

In the case of the allometric exponent $\eta=1$ the scaling limit of the population point processes $\nu^K$  has a richer structure of a nonlinear stochastic superprocess which is described below.
\bigskip

\begin{thm}
\label{readifstoch}
(i) Suppose that Assumptions (H1) and (H2)  are satisfied, $\eta=1$  and that $\tilde{\sigma}$ is  continuous. Additionally, assume   that, as $K\to \infty$, the
initial conditions $\nu^K_0$ converge in law, and for the weak topology
on $M_F$, to a finite (possibly random)
measure $X_0$, and that
\begin{equation}
   \label{eq:X4}
 \sup_K \E(\langle \nu^K_0,1\rangle^4)<+\infty.
\end{equation}
Then, for each  $T>0$, the  laws of the processes $\nu^K\in \dit([0,T],M_F)$ are tight  and the limiting values are  
superprocesses $X \in C([0,T],M_F)$  satisfying the following two 
conditions:
\begin{equation}
  \label{eca}
   \sup_{t \in [0,T]}\E \left( \langle X_t,1\rangle^4 \right) <\infty,
\end{equation}
and, for any $f \in C^2_b(\mathbb{R})$,
\begin{align}
  \label{dfmf}
  \bar M^f_t&=\langle X_t,f\rangle - \langle X_0,f\rangle - 
  \int_0^t  \int_{\rit}  \tilde\sigma(x)\,D^\alpha f(x)
 X_s(dx)ds \notag
  \\ &- \int_0^t \int_{\rit} f(x)\left(
    b(x,V*X_s(x))-d(x,U*X_s(x))\right)X_s(dx)ds
\end{align}
is a continuous martingale with the quadratic variation

\begin{equation}
  \label{qvmf}
  \langle\bar M^f\rangle_t= 2\int_0^t \int_{\rit}r(x) f^2(x) X_s(dx)ds.
\end{equation}
(ii) Assume moreover that $\hat{\sigma}$ is Lipschitz continuous and that $r$ is bounded from below by a positive constant. Then there is a unique such limiting superprocess.
\end{thm}

%\begin{rema} 
\bigskip
\noindent {\it Remark 3.4.} Here again, as in the case of Theorem \ref{readif}, Theorem \ref{readifstoch} yields the existence of 
a measure-valued process $X$ which is a 
weak solution of the stochastic partial differential equation with fractional diffusion operator
  \begin{equation*}
      \label{eq:EDPS}
    \partial_t X_t(x) =\bigg(b(x,V*X_t(x))-d(x,U*X_t(x))\bigg)X_t(x)
   +   D^{\alpha}(\tilde{\sigma}\, X_t)(x)+ \dot{M}_t,
  \end{equation*}
  where $\dot{M}_t$ is a random fluctuation term reflecting the
  demographic stochasticity of this fast birth-and-death process;  the process
is  faster than the accelerated birth-and-death process which
  led to the deterministic reaction-diffusion approximation~(\ref{reacdiff}).
%\end{rema}
%%%%%%%%%%%%%%%%%%%%%%%%%%%%%%%%%%%%%%%%%%%%%%%%%%%%%%%%%%%%%%%%%%%%%

\section{Auxiliary Lemmas}

In this section we provide auxiliary lemmas needed in the proofs of Theorems   \ref{readif}, and \ref{readifstoch}. The latter will be completed in Section 5. We begin with 
a lemma which gives uniform estimates for the moments of the population process. Its proof  can be easily adapted from the the proof of an analogous result in   \cite{FM04} and is thus omitted.

\begin{lem}
\label{mep}
Assume that    $p\geq 2$, and $\: \sup_K \E(\langle \nu^K_0,1\rangle^p)<+\infty$. Then 
\begin{equation}
  \label{moment1}
\sup_{K}  \left(\E\Big( \sup_{t \in [0,T]}  \langle \nu^K_t,1\rangle^p\Big) \right) <\infty.
\end{equation}
\end{lem}

To provide mass control for $\nu^K$, we first need to study the action of the fractional Laplacian on functions 
 $$
 f_n(x)=\psi(0\vee(|x|-(n-1))\wedge 1),\qquad n\in{\mathbb N}, 
 $$ 
 where the function $\psi(x)=6x^5-15 x^4+10x^3\in C^2$. Observe that $\psi$  is non-decreasing on $[0,1]$, and such that 
 $$
 \psi(0)=\psi'(0)=\psi''(0)=1-\psi(1)=\psi'(1)=\psi''(1)=0.
 $$

\begin{lem}\label{propfn}
For each $n\in{\mathbb N}^*$, the function $f_n$ is in  $C^2$, even, non-decreasing on $\rit_+$, equal to $0$ on $[-(n-1),n-1]$, and to $1$ 
on $(-n,n)^c$. In particular $f_0\equiv 1$. Moreover, 
$$
\sup_{n\in\N^*,x\in\R}|D^\alpha f_n(x)| <+\infty,
$$
 and, for each $  n\in\N^*,$ and $x\in(-n+1,n-1),$
$$
\;|D^\alpha f_n(x)|\leq \frac{2}{\alpha(n-1-|x|)^\alpha}.
$$
Finally, under Assumption { (H2)}, 
$$
\lim_{K\to \infty} K^\eta \int_\rit (f_n(x+h)-f_n(x))M_K(x,dh)=\sigma(x)D^\alpha f_n(x),
$$
 uniformly, for $n\in\N^*$, and $x\in\rit$.
\end{lem}

\begin{proof}
%[of Lemma \ref{propfn}]
For $n\in{\mathbb N}^*$, let us check the three last statements for $f_n$, the others being obvious. Using the Taylor expansion for $|y|\leq 1$ and remarking that, for each $ x,z\in\rit$, $|f_n(z)-f_n(x)|\leq 1$, one has
\begin{align*}
 |D^\alpha f_n(x)|&\leq \int_{|y|\leq 1}\frac{\sup_{z\in[0,1]}|\psi''(z)|y^2}{2}\times\frac{dy}{|y|^{1+\alpha}}+\int_{|y|\geq 1}\frac{dy}{|y|^{1+\alpha}}\\
&\leq \frac{\sup_{z\in[0,1]}|\psi''(z)|}{2-\alpha}+\frac{2}{\alpha}.
\end{align*}
For $x\in(-n+1,n-1)$, since $f_n(x)=f'_n(x)=0$, and $f_n(x+y)=0$ if $|y|\leq n-1-|x|$, 
$$|D^\alpha f_n(x)|=\left|\int_\rit \frac{f_n(x+y)dy}{|y|^{1+\alpha}}\right|\leq \int_{|y|>n-1-|x|}\frac{dy}{|y|^{1+\alpha}}=\frac{2}{\alpha(n-1-|x|)^\alpha}.$$
Since $\sup_{n\in\N^*,x\in\R}|f_n''(x)|=\sup_{x\in[0,1]}|\psi''(x)|<+\infty$, writing 
\eqref{contrest} for the function $f_n$ we see that to establish the uniform convergence we only need to check that $\sup_{n\in\N^*,x\in\R}\int_1^{+\infty}|f'_n(x+z)-f'_n(x-z)|\frac{dz}{2z^\alpha}<+\infty$. Since $\sup_{n\in\N^*,x\in\R}|f_n'(x)|=\sup_{x\in[0,1]}|\psi'(x)|<+\infty$, and $f'_n$ vanishes outside $[-n,-n+1]\cup[n-1,n]$, $$\sup_{n\in\N^*,x\in\R}\int_1^{+\infty}|f'_n(x+z)-f'_n(x-z)|\frac{dz}{2z^\alpha}\leq \sup_{x\in[0,1]}|\psi'(x)|\int_1^2\frac{dz}{z^\alpha},$$ which concludes the proof.
\end{proof} 

\bigskip

The next lemma provides   mass control for the sequence $(\nu^K)$.

\begin{lem}
\label{controle-masse}Under the assumptions of Theorem \ref{readif} or Theorem \ref{readifstoch}, 
$$\lim_{n\to\infty}\limsup_{K\to\infty}\E\left(\sup_{t\leq T}\langle \nu^K_t,f_n\rangle\right)=0.$$
\end{lem}

\bigskip

\begin{proof}
%[ of Lemma \ref{controle-masse}]
 Let $M^{K,n}_t$ denote the square integrable martingale defined by \eqref{mart1} with $f_n$ replacing $f$. The boundedness of $r$ and $\sigma$ together with the last assertion in Lemma \ref{propfn} ensures the existence of a sequence $(\varepsilon_K)_K$ converging to $0$ such that
\begin{align}
 \langle \nu^K_t,f_n\rangle&\leq \langle \nu^K_0,f_n\rangle + M_t^{K,n}+\bar{b}\int_0^t\langle \nu^K_s,f_n\rangle ds+\varepsilon_K\int_0^t\langle \nu^K_s,1\rangle ds\notag\\&\hskip2cm+\int_0^t\int_\R r(x)p(x)\sigma(x)D^\alpha f_n(x)\nu^K_s(dx)ds.
\end{align}
For $n\geq m>1$, splitting $\mathbb{R}$ into $(-n+m, n-m)$ and its complement, and  using Lemma \ref{propfn},  we get
\be
\label{recfn}
\int_\R&& r(x)p(x)\sigma(x)D^\alpha f_n(x)\nu^K_s(dx)\notag\\
&&\leq \bar{r}\sup_{x\in\rit}\sigma(x)\left(\sup_{l,x}|D^\alpha f_l(x)|\langle \nu^K_s,f_{n-m}\rangle+\frac{2}{\alpha (m-1)^\alpha}\langle \nu^K_s,1\rangle\right).
\ee

\noindent Since the sequence $(f_n)_n$ is non-increasing,  $\langle \nu^K_s,f_n\rangle \leq \langle \nu^K_s,f_{n-m}\rangle$, and
there is a constant $C$ not depending on $n\geq 2$ and $m\in\{2,\hdots,n\}$, and a sequence $(\eta_m)_{m\geq 2}$ of positive numbers converging to $0$ such that
\begin{align}
 \langle \nu^K_t,f_n\rangle&\leq \langle \nu^K_0,f_n\rangle + M_t^{K,n}+C\int_0^t\langle \nu^K_s,f_{n-m}\rangle ds+\left(\varepsilon_K+\eta_m\right)\int_0^t\langle \nu^K_s,1\rangle ds.\label{majomu}\end{align}
Let $\mu^{K,n}_t=E\left(\sup_{s\leq t}\langle \nu^K_s,f_n\rangle\right)$ and $\mu^K_t=E\left(\sup_{s\leq t}\langle \nu^K_s,1\rangle\right)$ which is bounded uniformly in $K$ and $t\in[0,T]$ since, 
 using Lemma \ref{mep}, 
\begin{equation}
\label{me2}
\sup_K \E \big(\sup_{t \in [0,T]} \langle \nu^K_t,1\rangle^3\big)
<\infty.
\end{equation}

\medskip
 Assume, first,  that $0<\eta<1$. Observing  that, in view of  \eqref{quadra1},
$\langle M^{K,n}\rangle_t\leq CK^{\eta-1}\int_0^t\langle \nu^K_s,1\rangle+\langle \nu^K_s,1\rangle^2 ds$, and using Doob's inequality and \eqref{me2}, one deduces that
\begin{align}
\mu^{K,n}_t\leq \mu^{K,n}_0+C\int_0^t\mu^{K,n-m}_s ds+\varepsilon_K+\eta_m\label{itermu}\end{align}
for the modified sequences $(\varepsilon_K)_K$ and $(\eta_m)_m$ that still converge to $0$, as $K$ and $m$, respectively, grow to $\infty$. Iterating  this inequality yields,  for $j\in\N^*$, and $m>1$, 
\begin{align*}
 \mu^{K,jm}_T&\leq \sum_{l=0}^{j-1}\mu^{K,(j-l)m}_0\frac{(CT)^l}{l!}+\frac{C^{j}T^{j-1}\int_0^T\mu^K_sds}{(j-1)!}+(\varepsilon_K+\eta_m)\sum_{l=0}^{j-1}\frac{(CT)^l}{l!}\\
&\leq \mu^{K,\lfloor j/2\rfloor m}_0e^{CT}+E\left(\langle \nu^K_0,1\rangle\right)\sum_{l=\lfloor (j+3)/2\rfloor}^{+\infty}\frac{(CT)^l}{l!}+\frac{C'(CT)^{j}}{(j-1)!}+(\varepsilon_K+\eta_m)e^{CT},
\end{align*}
where we used the monotonicity of $\mu^{K,n}_0$ w.r.t. $n$, and \eqref{me2}  to justify  the second inequality. 

The random variables $(\langle \nu^K_0,f_n\rangle)_K$ converge in law to $\langle \xi_0,f_n\rangle$, as $K\to\infty$, and are uniformly integrable according to \eqref{eq:X3}. Therefore,  for a fixed $n$, $\mu^{K,n}_0$ converges to $\langle \xi_0,f_n\rangle$, as $K\to\infty$.
Hence, 
$$\limsup_{K\to\infty}\mu^{K,jm}_T\leq \langle \xi_0,f_{\lfloor j/2\rfloor m}\rangle e^{CT}+\sup_K \E\left(\langle \nu^K_0,1\rangle\right)\sum_{l=\lfloor (j+3)/2\rfloor}^{+\infty}\frac{(CT)^l}{l!}+\frac{C'(CT)^{j}}{(j-1)!}+\eta_me^{CT}.$$
 For $m,j$ large enough, the right-hand-side is arbitrarily small. Since $n\mapsto \mu^{K,n}_T$ is non-increasing, the proof is complete in the case $0<\eta<1$.

\medskip
 When $\eta=1$, $\sup_n\E(\sup_{t\leq T}\langle M^{K,n}\rangle_t)$ does not vanish anymore as $K\to\infty$, and the proof cannot be concluded   in the same way as above. But, taking expectations in \eqref{majomu} we obtain that \eqref{itermu} holds with $\sup_{s\leq t}\E(\langle \nu^K_s,f_n\rangle)$, $\E(\langle \nu^K_0,f_n\rangle)$ and $\sup_{r\leq s}\E(\langle \nu^K_r,f_{n-m}\rangle )$, replacing,  respectively,  $\mu^{K,n}_t$, $\mu^{K,n}_0$, and $\mu^{K,n-m}_s$. Following the previous line of reasoning we obtain that $$\lim_{n\to\infty}\limsup_{K\to\infty}\sup_{t\leq T}\E(\langle \nu^K_t,f_n\rangle)=0.$$
Now,  in view of  \eqref{quadra1},  and because $f_n^2\leq f_n$, one gets
\begin{align*}
 \langle M^{K,n}\rangle_t\leq\bar{r}\int_0^t\left(2\langle \nu^K_s,f_n\rangle+\int_{\rit^2}f_n(x+h)M_K(x,dh)\nu^K_s(dx)\right)ds+\varepsilon_K,
\end{align*}
with $(\varepsilon_{K})_{K}$ tending to $0$ as $K$ converges to infinity. 
Since by  \eqref{heavytail}, for $K$ large enough, 
 \begin{align*}
 \int_{\rit}f_n(x+h)M_K(x,dh)&=\E\left(f_n\Big(x+\frac{X(x)}{K^{1/\alpha}}\Big)\right)\leq f_{n-1}(x)+P\Big(\frac{X(x)}{K^{1/\alpha}}>1\Big)\\&\leq f_{n-1}(x)+\frac{1}{K}\Big(\frac{2\sup_{y\in\rit}\sigma(y)}{\alpha}+1\Big),
\end{align*}
one deduces 
  that $\lim_{n\to\infty}\limsup_{K\to\infty}E(\langle M^{K,n}\rangle_T)=0$. 
Taking advantage of  \eqref{majomu} and Doob's inequality one checks that \eqref{itermu} holds with $\varepsilon_{K}$ replaced by $\varepsilon_{K,n}$ such that $\lim_{n\to\infty}\limsup_{K\to\infty}\varepsilon_{K,n}=0$ on  the right-hand-side. Then one easily adapts the end of the proof written for $\eta<1$ to obtain the desired conclusion.
\end{proof}

\bigskip
\begin{rema} \label{rem-contmass} In the case of a mutation kernel  with finite second order moment as in \cite{CFM08}, the true Laplacian being a local operator, \eqref{recfn} may be replaced by 
\ben
\int_\R r(x)p(x)\sigma(x) f''_{n}(x)\nu^K_s(dx)
\leq \bar{r}\sup_{x\in\rit}\sigma(x)\, \sup_{l,x}| f''_l(x)|\, \langle \nu^K_s,f_{n-1}\rangle.
\een
The conclusion of Lemma \ref{controle-masse} may be derived by a similar argument. 
\end{rema}

Now, let $N(dt,dh)$ be a Poisson random measure on $\rit_+\times (-1,1)$ with  intensity $dt\frac{dh}{|h|^{1+\alpha}}$. The process $Z_t=\int_{(0,t]\times (-1,1)}h\Big(N(dt,dh)-\frac{dh}{|h|^{1+\alpha}}\Big)$ is a L\'evy process such that, for all $\, p\geq 1$, $\E(|Z_1|^p)<+\infty$. 
Let $X^x_t$ denote the solution of  the stochastic differential equation 
\begin{equation}
dX^x_t=\hat{\sigma}(X^x_{t^-}) dZ_t,\;X^x_0=x  
\label{edspetitsauts}
\end{equation}
which admits a unique solution when $\hat{\sigma}$ is Lipschitz continuous according to  Theorem 7, p.259, in \cite{Protter}. Let us denote by $P_t$ the associated semigroup defined  for all measurable and bounded  $f:\rit\rightarrow \rit$ by  the formula $P_tf(x)=\E(f(X^x_t))$. The infinitesimal generator of the process $X^x$ is 
\be
\label{defl}
Lf(x)=\int_{(-1,1)}\Big(f(x+\hat{\sigma}(x)h)-f(x)-f'(x)\hat{\sigma}(x)h\Big)\frac{dh}{|h|^{1+\alpha}}.
\ee

\begin{lem}\label{regukolmog}
Assume that $\hat{\sigma}$ is $C^2$ with a bounded first-order derivative, and a bounded, and locally Lipschitz second-order derivative. For $f\in C^2_b(\rit)$, $P_tf(x)$ belongs to $C^{1,2}_b([0,T]\times \rit)$, and solves the initial-value problem
\begin{equation*}
\begin{cases}
 \partial_t P_tf(x)=LP_tf(x),\;(t,x)\in[0,T]\times\rit\\
 P_0f(x)=f(x),\;x\in\rit
\end{cases}.
\end{equation*}
\end{lem}

\bigskip

\begin{proof}
%[of Lemma  \ref{regukolmog}]
By Theorem 40, p.317. in  \cite{Protter}, the mapping $x\mapsto X^x_t$ is twice continuously differentiable with first, and second order derivatives  solving, respectively,  the equations
$$
 d\partial_x X^x_t=\hat{\sigma}'(X^x_{t^-})\partial_x X^x_{t^-}dZ_t,\;\partial_x X^x_0=1, 
$$
and 
$$
 d\partial_{xx} X^x_t=\hat{\sigma}'(X^x_{t^-})\partial_{xx} X^x_{t^-}dZ_t+\hat{\sigma}''(X^x_{t^-})(\partial_x X^x_{t^-})^2dZ_t,\;\partial_{xx} X^x_0=0.
$$
For $q\geq 1$, since $\E(|Z_T|^q)<+\infty$, by Theorem 66,  p.346, in \cite{Protter}, there is a finite constant $K$ such that, for any predictable process $(H_t)_{t\leq T}$, and any  $ t\leq T,$\;\begin{equation}
\E\left(\sup_{s\leq t}\left|\int_0^sH_rdZ_r\right|^q\right)\leq K\int_0^tE(|H_s|^q)ds.\label{momz}
\end{equation}
This, combined with the regularity assumptions made on $\hat{\sigma}$, and Gronwall's Lemma, immediately implies that, for any $ q\geq 1,$
\begin{equation}
 \sup_{(t,x)\in[0,T]\times\rit}\E(|\partial_x X^x_t|^q+|\partial_{xx} X^x_t|^q)<+\infty.\label{momder}
\end{equation}
For $f\in C^2_b(\rit)$, one deduces that the mapping $x\mapsto P_tf(x)\in C^2_b(\rit)$,  with $\partial_xP_tf(x)=\E(f'(X^x_t)\partial_x X^x_t)$ and $\partial_{xx}P_tf(x)=\E(f''(X^x_t)(\partial_x X^x_t)^2+f'(X^x_t)\partial_{xx}X^x_t)$. Indeed, for instance, to differentiate for the second time under the expectation, we observe  that, as $y$ tends to $x$, the random variables $ ({f'(X^y_t)\partial_x X^y_t-f'(X^x_t)\partial_x X^x_t})/({y-x}) $ converge a.s. to $f''(X^x_t)(\partial_x X^x_t)^2+f'(X^x_t)\partial_{xx}X^x_t$, and  by \eqref{momder} are uniformly integrable since
\begin{align*}
 \left|\frac{f'(X^y_t)\partial_x X^y_t-f'(X^x_t)\partial_x X^x_t}{y-x}\right|&\leq \left|f'(X^y_t)\frac{\partial_x X^y_t-\partial_x X^x_t}{y-x}\right|+\left|\frac{f'(X^y_t)-f'(X^x_t)}{y-x}\right||\partial_x X^x_t|\\
&\leq \frac{\|f'\|_\infty}{y-x}\int_x^y|\partial_{xx} X^z_t|dz+|\partial_x X^x_t|\frac{\|f''\|_\infty}{y-x}\int_x^y |\partial_x X^z_t|dz.
\end{align*}
Moreover, since $f(X^x_t)$, $f'(X^x_t)\partial_x X^x_t$,  and $f''(X^x_t)(\partial_x X^x_t)^2+f'(X^x_t)\partial_{xx}X^x_t$, are continuous w.r.t. $x$, and right-continuous and quasi left-continuous  w.r.t. $t$, one deduces that the mapping $(t,x)\mapsto (P_tf(x),\partial_x P_tf(x),\partial_{xx}P_tf(x))$ is continuous and bounded on $[0,T]\times\rit$. With the upper-bound,
$$|P_tf(x+\hat{\sigma}(x)h)-P_tf(x)-\partial_x P_tf(x)\hat{\sigma}(x)h|\leq \frac{1}{2}\|\partial_{xx}P_tf\|_\infty\hat{\sigma}^2(x) h^2,$$ 
one concludes that the mapping $(t,x)\mapsto LP_tf(x)$ is continuous and bounded on $[0,T]\times\rit$.
For $u>0$, by the Markov property stated in Theorem 32, p.300, of  \cite{Protter}, $P_{t+u}f(x)=\E(P_tf(X^x_u))$. By It\^o's formula,
\begin{align*}
 P_tf(X^x_u)=&P_tf(x)+\int_0^u LP_tf(X^x_s)ds\\&+\int_{(0,u]\times (-1,1)}\left(P_tf(X^x_{s^-}+\hat{\sigma}(X^x_{s^-})h)-P_tf(X^x_{s^-})\right) (N(ds,dh)-\frac{dh}{|h|^{1+\alpha}}).
\end{align*}
Since the last integral is a martingale, it is centered and one obtains that $$\frac{P_{t+u}f(x)-P_tf(x)}{u}=\E\left(\frac{1}{u}\int_0^u LP_tf(X^x_s)ds\right).$$
By Lebesgue's Theorem, one deduces that $\lim_{u\to 0^+}\frac{P_{t+u}f(x)-P_tf(x)}{u}=LP_tf(x)$. Hence $(t,x)\mapsto P_tf(x)$ belongs to $C^{1,2}([0,T]\times \rit)$ and solves the initial-value problem
\begin{equation*}
\begin{cases}
 \partial_t P_tf(x)=LP_tf(x),\;(t,x)\in[0,T]\times\rit\\
 P_0f(x)=f(x),\;x\in\rit
\end{cases}.
\end{equation*}
\end{proof}

\medskip

 Unfortunately, in the general case needed in the proofs of our theorems,  $\hat{\sigma}$ is merely Lipschitz, and $P_{t-s}f(x)$ is not smooth enough in the spatial variable $x$. That is why, for $\varepsilon>0$, we set $\hat{\sigma}^\varepsilon(x)=\int_{\R}\hat{\sigma}(x-y)e^{-\frac{y^2}{2\varepsilon}} {dy}/{\sqrt{2\pi\varepsilon}}$, and define $X^{\varepsilon,x}_t$ as the solution of the SDE similar to \eqref{edspetitsauts}, but with $\hat{\sigma}$ replaced by $\hat{\sigma}^\varepsilon$. The generator of $X^{\varepsilon,x}_t$ is the operator $L^\epsilon$ defined like $L$, but with $\hat{\sigma}^\varepsilon$ replacing $\hat{\sigma}$. Finally,  we set $P^\varepsilon_t f(x)=\E(f(X^{\varepsilon,x}_t))$.  Now, we want to know what happens when   $\varepsilon$ tends to $0$.  
The next lemma  gives the H\"older's continuity of $\partial_xP_{t-s}^\varepsilon f$ and the order of convergence of $P^\varepsilon_tf(x)$ to $P_tf(x)$.

\begin{lem}\label{contderpeps}
  Assume that $\hat{\sigma}$ is Lipschitz, and $f\in C^2_b(\R)$. Then there exists a constant $  C>0$, such that for all $t\in[0,T],$ and $x,y\in\R $ such that 
  $|x-y|\leq \|\hat{\sigma}\|_\infty,$ 
$$
|\partial_xP_{t-s}^\varepsilon f(x)-\partial_xP_{t-s}^\varepsilon f(y)|\leq  \frac{C|x-y|^{\alpha/2}}{\varepsilon^{\alpha/4}}.$$
Moreover, 
\be
\label{cvsg}
\sup_{(t,x)\in[0,T]\times \R}|P_tf(x)-P^\varepsilon_tf(x)|\leq C\sqrt{\varepsilon}.
\ee
\end{lem}

\begin{proof}
%[of Lemma \ref{contderpeps}]
One has
\begin{align*}
\partial_xP^\varepsilon_t f(x)-\partial_xP^\varepsilon_t f(y) =&\E\left(f'(X^{x,\varepsilon}_t)(\partial_xX^{x,\varepsilon}_t-\partial_xX^{y,\varepsilon}_t)\right)\\
& +\int_y^x\E\left(f''(X^{z,\varepsilon}_t)\partial_xX^{z,\varepsilon}_t\partial_xX^{y,\varepsilon}_t\right)dz.
\end{align*}
The absolute value of the second term of the r.h.s. is smaller than $C|y-x|$ since, for each $q \ge 1$, 
\begin{equation}  
\sup_{\varepsilon>0}\sup_{(t,x)\in[0,T]\times\rit}\E(|\partial_x X^{x,\varepsilon}_t|^q)<+\infty,
\label{momdereps}
\end{equation} 
because $\sup_{\varepsilon>0}\|{{\hat{\sigma}}^{\varepsilon}}\,'\|_\infty$ is not greater than the Lipschitz constant of $\hat{\sigma}$.
To deal with the first term, we remark that 
\begin{align*}
\partial_xX^{x,\varepsilon}_t-\partial_xX^{y,\varepsilon}_t=&\int_0^t{{\hat\sigma}^\varepsilon}\,'(X^{x,\varepsilon}_{s^-})(\partial_xX^{x,\varepsilon}_{s^-}-\partial_xX^{y,\varepsilon}_{s^-})dZ_s\\
&+\int_0^t({{\hat{\sigma}}^\varepsilon}\,'(X^{x,\varepsilon}_{s^-})-{{\hat{\sigma}}^\varepsilon}\,'(X^{y,\varepsilon}_{s^-}))\partial_xX^{y,\varepsilon}_{s^-}dZ_s
\end{align*}
Using \eqref{momz},  and the inequality $\|{{{\hat{\sigma}}^\varepsilon}}\,''\|_\infty\leq \frac{C}{\sqrt{\varepsilon}}$, one deduces that
\begin{align*}
  \E\left(\sup_{s\leq t}(\partial_xX^{x,\varepsilon}_s-\partial_xX^{y,\varepsilon}_s)^2\right)\leq &C\int_0^t\E((\partial_xX^{x,\varepsilon}_s-\partial_xX^{y,\varepsilon}_s)^2)ds\\&+\frac{C(x-y)}{\varepsilon}\int_0^t\int_y^x\E\left((\partial_xX^{z,\varepsilon}_{s})^2(\partial_xX^{y,\varepsilon}_{s})^2\right)dzds.
\end{align*}
In view of  \eqref{momdereps}, and Gronwall's Lemma,  one concludes that
\begin{align}
  \E\left(\sup_{s\leq T}(\partial_xX^{x,\varepsilon}_s-\partial_xX^{y,\varepsilon}_s)^2\right)\leq\frac{C(x-y)^2}{\varepsilon}\label{contecflot}.
\end{align}
Combining this bound with the inequality 
\begin{align*}
 |\E\left(f'(X^{x,\varepsilon}_t)(\partial_xX^{x,\varepsilon}_t\!-\partial_xX^{y,\varepsilon}_t)\right)|&\leq C\E\left(|\partial_xX^{x,\varepsilon}_t-\partial_xX^{y,\varepsilon}_t|^{\frac{\alpha}{2}} (|\partial_xX^{x,\varepsilon}_t|+|\partial_xX^{y,\varepsilon}_t|)^{1-\frac{\alpha}{2}}\right),
\end{align*}
and H\"older's inequality, one easily deduces the first statement of the Lemma. 
Since $\|\hat{\sigma}-\hat{\sigma}^\varepsilon\|_\infty\leq C\sqrt{\varepsilon}$, by a reasoning similar to the one employed  to prove \eqref{contecflot}, one easily checks that
$\E\left(\sup_{s\leq T}(X^{x}_s-X^{x,\varepsilon}_s)^2\right)\leq C\varepsilon$ and the second statement follows.

\end{proof}

%%%%%%%%%%%%%%%%%%%%%%%%%%%%%%%%%%%%%%%%%%%%%%%%%%%%%%%%%%%%%%%%%%%%%%%%

\section{Proofs of the Theorems }

\noindent 
\paragraph{Proof of Theorem~\ref{readif}}

\noindent  The proof of the theorem will be carried out in five steps. Let us fix $T>0$, and $\eta<1$.

\medskip
\noindent {\bf Step 1} \phantom{9} We first endow $M_F$ with the vague
topology. To show the tightness of the sequence of laws
$Q^K=\loi(\nu^K)$ in ${\cal
P}(\dit([0,T],(M_F,v)))$, where $M_F$ is endowed with the vague convergence topology, it suffices, following Roelly~\cite{Ro86}, to
show that for any continuous bounded function $f$ on $\rit$ the
sequence of laws of the processes $\langle \nu^K,f\rangle$ is
tight in $\dit([0,T], \rit)$. To this end  we use the Aldous
criterion~\cite{Al78} and the Rebolledo criterion
(see~\cite{JM86}) which require us  to show that
\begin{equation}
\label{tight} \sup_K
\E\big(\sup_{t\in [0,T]} | \langle \nu^K_t,f\rangle | \big)
<\infty,
\end{equation}
and   the laws, respectively,  of the predictable
quadratic variation of the martingale part, and
of the drift part of the semimartingales $\langle \nu^K,f\rangle$, are tight. 

Since $f$ is bounded,~(\ref{tight}) is a consequence
of~(\ref{me2}): let us thus consider a pair $(S,S')$ of stopping
times satisfying a.s. the inequality $0 \leq S \leq S' \leq S+\delta\leq T$.
Using~(\ref{quadra1}) and (\ref{me2}), for some positive real numbers $C$ and $
C'$, we get 
\begin{equation*}
\E\left(\langle M^{K,f}\rangle_{S'}- \langle
  M^{K,f}\rangle_S\right) \leq C \E\left( \int_S^{S+\delta} \left(
    \langle \nu^K_s,1\rangle+ \langle \nu^K_s,1\rangle^2
  \right)ds\right)\leq C' \delta.
\end{equation*}
In a similar way,  we show that the expectation of the finite variation part of
$\langle \nu^K_{S'},f\rangle - \langle \nu^K_S,f\rangle$ is
bounded by $C' \delta$.
  Hence,  the sequence $Q^K=\loi(\nu^K)$ is tight in ${\cal
P}(\dit([0,T],(M_F,v)))$.

\bigskip
\noindent {\bf Step 2} \phantom{9} Let us now denote by $Q$ the weak limit in ${\cal
P}(\dit([0,T],(M_F,v)))$
of a subsequence of $(Q^K)$ which we  also denote 
$(Q^K)$. We remark
that by construction,
 \begin{equation*}
\sup_{t\in [0,T]}\ \sup_{f \in L^\infty(\rit), || f
  ||_\infty\leq 1} | \langle \nu^K_t,f \rangle -
\langle \nu^K_{t^-},f \rangle | \leq 1/K.
\end{equation*}
Since, for each $f$ in a countable measure-determining set of continuous and compactly supported functions on $\rit$, the mapping $\nu\mapsto\sup_{t\leq T} |\langle \nu_t,f\rangle - \langle \nu_{t-},f\rangle |$ is continuous on $\mathbb{D}([0,T], (M_F,v))$, one deduces that $Q$ only charges the continuous processes  from $[0,T]$ into $(M_F,v)$. 
Let us now endow  $M_F$ with the weak convergence topology and check that $Q$ only charges the continuous processes from $[0,T]$ into $(M_F,w)$, and that the sequence $(Q^K)$ in ${\cal
P}(\dit([0,T],(M_F,w))$ converges weakly to $Q$.
 For this purpose, we need to control the behavior of the total mass of the measures. We will employ  the sequence $(f_n)$ of smooth functions introduced in  Lemma  \ref{controle-masse} which approximate the functions ${\bf 1}_{\{|x|\geq n\}}$.
For each  $n\in{\mathbb N}$, the continuous and compactly supported functions $(f_{n,l}\stackrel{\rm def}{=}f_n(1-f_l))_{l\in{\mathbb N}}$ increase to $f_n$, as $l\to\infty$. Continuity of the mapping $\nu\mapsto \sup_{t\leq T} \langle \nu_t,f_{n,l}\rangle$ on $\mathbb{D}([0,T], (M_F,v))$, and its    uniform integrability deduced from 
\eqref{me2}, imply the bound
$$
\E^Q\left(\sup_{t\leq T}\langle \nu_t,f_{n,l}\rangle\right)=\lim_{K\to\infty}\E\left(\sup_{t\leq T}\langle \nu^K_t,f_{n,l}\rangle\right)\leq\liminf_{K\to\infty}\E\left(\sup_{t\leq T}\langle \nu^K_t,f_{n}\rangle\right).
$$
Taking the limit, $l\to\infty$,  in the left-hand-side, in view of the monotone  convergence theorem and respectively, \eqref{me2} and Lemma \ref{controle-masse}, one concludes that   for $n=0$, 
\begin{equation}
 \E^Q\left(\sup_{t\leq T}\langle\nu_t,1\rangle\right)=\E^Q\left(\sup_{t\leq T}\langle\nu_t,f_0\rangle\right)<+\infty\label{contmass}
\end{equation} 
 and for general $n$,    \begin{equation}
 \lim_{n\to\infty}\E^Q\left(\sup_{t\leq T}\langle \nu_t,f_n\rangle\right)=0.\label{tightlim}
\end{equation}%%%%%%%%%
As a consequence one may extract a subsequence of the sequence $(\sup_{t\leq T}\langle \nu_t,f_n\rangle)_n$ that converges a.s. to $0$ under $Q$, and the set $(\nu_t)_{t\leq T}$ is tight  $Q$-a.s. Since 
$Q$ only charges the continuous processes  from $[0,T]$ into $(M_F,v)$, one deduces that $Q$ also only charges  the continuous processes  from $[0,T]$ into $(M_F,w)$.

Let $X$ denote a process with law $Q$. According to M\'el\'eard and Roelly~\cite{MR93}, to prove that the sequence $(Q^K)$ converges weakly to $Q$ in ${\cal
P}(\dit([0,T],(M_F,w))$, it is sufficient to check that the processes $(\langle\nu^K,1\rangle=(\langle \nu^K_t,1\rangle)_{t\leq T})_K$ converge in law to $\langle X,1\rangle\stackrel{\rm def}{=}(\langle X_t,1\rangle)_{t\leq T}$ in $\dit([0,T],\rit)$.
For  a Lipschitz continuous and bounded function $F$ from  $\mathbb{D}([0,T],  \mathbb{R})$ 
to $\mathbb{R}$, we have
\begin{align*}
\limsup_{K\to\infty}|\E(F(\langle \nu^K,1\rangle)&- F(\langle X,1\rangle)| \leq \limsup_{n\to\infty}\, \limsup_{K\to\infty}|\E(F(\langle \nu^K,1\rangle)- F(\langle \nu^K,1-f_n\rangle))| \\
&+
\limsup_{n\to\infty}\limsup_{K\to\infty}|E(F(\langle \nu^K,1-f_n\rangle)- F(\langle X,1-f_n\rangle))| \\
&+\limsup_{n\to\infty}|\E(F(\langle X,1-f_n\rangle)- F(\langle X,1\rangle))|. 
\end{align*}
Since $|F(\langle \nu,1-f_n\rangle)- F(\nu ,1\rangle)|\leq C\sup_{t\leq T}\langle \nu_t,f_n\rangle$, Lemma \ref{controle-masse} and  \eqref{tightlim} respectively imply that the first and the third terms in the r.h.s. are equal to $0$. The second term is $0$ in view of the  continuity of the mapping $\nu\mapsto \langle\nu,1-f_n\rangle$ in $\dit([0,T],(M_F,w))$.

\bigskip
\noindent
{\bf Step 3} \phantom{9} Recall that the time $T>0$ is  fixed, and $0<\eta<1$. Let us  check
that, almost surely, the process $X$ solves~(\ref{readif1}). By \eqref{contmass},   for each $T$, 
$\sup_{t\in
[0,T]} \langle X_t,1\rangle   $ is finite a.s. Now, we fix 
a function $f\in C^2_b(\rit)$, compactly supported if $\alpha\leq 1$ (in this case the extension of~(\ref{readif1})
to
any function $f$ in $C^2_b$ follows by Lebesgue's theorem), and a  $t\leq T$.\\
For $\nu \in \dit([0,T],(M_F,w))$, denote  
\begin{align}
\Psi^1_t(\nu)&= \langle\nu_t,f\rangle - \langle\nu_0,f\rangle -
\intot \int_{\rit}
(b(x,V*\nu_s(x))-d(x,U*\nu_s(x)))f(x)\nu_s(dx)ds,\notag \\
\Psi^2_t(\nu)&=- \int_0^t\int_{\rit}\tilde\sigma(x)D^\alpha f(x)\nu_s(dx)ds.
\end{align}
We must show that
\begin{equation}
\label{wwhtp}
\E^Q \left( |\Psi^1_t(X)+\Psi^2_t(X) | \right)=0.
\end{equation}

\noindent By~(\ref{mart1}), we know that, for each $K$,
\begin{equation*}
M^{K,f}_t=\Psi^1_t(\nu^K)+\Psi^{2,K}_t(\nu^K),
\end{equation*}
where
\begin{align*}
\Psi^{2,K}_t(\nu^K)
=&-\int_0^t\int_{\rit}p(x)(K^{\eta}r(x)+b(x,V*\nu^K_s(x)))\\
&\hskip 1cm \times \left( \int_{\rit}
f(x+h)M_K(x,h)dh-f(x)\right) \nu^K_s(dx)ds.
\end{align*}
Moreover,~(\ref{me2}) implies that for each $K$,
\begin{equation}
\label{cqmke}
\E \left( | M^{K,f}_t |^2 \right) = \E \left( \langle
    M^{K,f}\rangle_t \right)\leq \frac{C_{f}K^{\eta}}{K}E\left(
  \intot  \left\{\langle \nu^K_s,1\rangle
    +\langle \nu^K_s,1\rangle^2\right\}ds \right) \leq
\frac{C_{f,T}K^{\eta}}{K},
\end{equation}
which goes to $0$ as $K$ tends to infinity, since $0<\eta<1$. Since, by Proposition \ref{restlapfrac}, there exists a deterministic sequence $(\varepsilon_K)$, converging to $0$ as $K\to \infty$,  such that 
$$
|\Psi^{2,K}_t(\nu^K)-\Psi^{2}_t(\nu^K)|\leq \varepsilon_K \int_0^t\langle\nu^K_s,1\rangle ds,
$$
one deduces that
\begin{equation*}
\lim_K E(|\Psi^1_t(\nu^K)+\Psi^{2}_t(\nu^K)|)=0.
\end{equation*}
Since $X\in C([0,T],(M_F,w))$ and $f\in C^2_b(\rit)$, 
due to the continuity of the parameters, the functions
$\Psi^1_t$, and $\Psi^2_t$ are a.s.~continuous at $X$. Furthermore,
for any $ \nu \in \dit([0,T],M_F)$,
\begin{equation}
|\Psi^1_t(\nu)+\Psi^2_t(\nu)| \leq C_{f} \left(\langle\nu_t,1\rangle+\langle\nu_0,1\rangle+\int_0^t\langle \nu_s,1\rangle+\langle \nu_s,1\rangle^2 ds\right).
\end{equation}
Hence, in view of (\ref{me2}),   the sequence
$(\Psi^1_t(\nu^K)+\Psi^2_t(\nu^K))_K$ is uniformly integrable, and
thus
\begin{eqnarray}
\label{cqv1}
\lim_K \E\left(|\Psi^1_t(\nu^K)+\Psi^2_t(\nu^K)|\right)
=\E\left(|\Psi^1_t(X)+\Psi^2_t(X)|\right)=0, 
\end{eqnarray}
which concludes the proof of  the first part of Theorem  \ref{readif}.

\bigskip
\noindent
{\bf Step 4} \phantom{9} In this step we will prove part {\it  (ii)} of Theorem  \ref{readif} which asserts  uniqueness  of the solution of (\ref{readif1}) under the additional assumption  that $\hat{\sigma}$ is Lipschitz.    
According to \eqref{chgtvar}, one has
\begin{align*}
\tilde{\sigma}(x) D^\alpha f(x)
=Lf(x)+\int_{\rit\setminus(-1,1)}\left(f(x+\hat{\sigma}(x)h)-f(x)\right)\frac{dh}{|h|^{1+\alpha}} , 
\end{align*}
where $L$ has been defined in \eqref{defl}.
 It is easy to prove that if $\xi$ is a
solution of~(\ref{readif1}) satisfying $\sup_{t\in [0,T]}
\langle\xi_t,1 \rangle <\infty$, then, for each test function
$\psi_t(x)=\psi(t,x) \in C^{1,2}_b(\rit_+\times \rit)$, one gets
\begin{align}
\label{readif2}
\langle\xi_t,\psi_t\rangle&=\langle \xi_0,\psi_0\rangle+\int_0^t\int_{\rit}(\partial_s \psi(s,x)+\tilde{\sigma}(x)D^\alpha\psi_s(x))\xi_s(dx)ds\notag\\ &\quad +
\int_0^t\int_{\rit}
(b(x,V*\xi_s(x))-d(x,U*\xi_s(x)))\psi(s,x)\xi_s(dx)ds \notag 
\\
&=\langle \xi_0,\psi_0\rangle+\int_0^t\int_{\rit}(\partial_s \psi(s,x)+L\psi_s(x))\xi_s(dx)ds
\notag \\ &
\quad +\quad \int_0^t\int_{\rit}
\bigg((b(x,V*\xi_s(x))-d(x,U*\xi_s(x)))\psi(s,x)\notag \\
&  \quad \phantom{+\int_0^t\int_{\rit}
\bigg(}+\int_{\rit\setminus(-1,1)}\left(\psi(s,x+\hat{\sigma}(x)h)-\psi(s,x)\right)\frac{dh}{|h|^{1+\alpha}}\bigg)\xi_s(dx)ds.
\end{align}

Let $t\in[0,T]$ and $f\in C^2_b(\rit)$. 
We would like to choose $\psi(s,x)=P_{t-s}f(x)$, where $P_tf(x)=E(f(X^x_t))$, and $X^x_t$ is the unique solution of the stochastic differential equation \eqref{edspetitsauts},   so that the second term in the right-hand-side above vanishes. 
In view of Lemma  \ref{regukolmog}, this  is immediately possible if $\hat{\sigma}$ is assumed to be $C^2$, with a bounded first order derivative, and a bounded and locally Lipschitz second order derivative.  
 
  Unfortunately, when  $\hat{\sigma}$ is merely Lipschitz continuous, $P_{t-s}f(x)$ is not smooth enough in the spatial variable $x$. That is why, for $\varepsilon>0$, we set $\hat{\sigma}^\varepsilon(x)=\int_{\R}\hat{\sigma}(x-y)e^{-\frac{y^2}{2\varepsilon}}\frac{dy}{\sqrt{2\pi\varepsilon}}$ and define $X^{\varepsilon,x}_t$ as the solution to the SDE similar to \eqref{edspetitsauts} but with $\hat{\sigma}$ replaced by $\hat{\sigma}^\varepsilon$. The generator of $X^{\varepsilon,x}_t$ is the operator $L^\epsilon$ defined like $L$ but with $\hat{\sigma}^\varepsilon$ replacing $\hat{\sigma}$ and we set $P^\varepsilon_t f(x)=\E(f(X^{\varepsilon,x}_t))$. According to   Lemma \ref{regukolmog},  for the choice $\psi(s,x)=P^\varepsilon_{t-s}f(x)$, equation \eqref{readif2} takes the form
\begin{align}
\label{readif3bis}
 \langle
\xi_t,f\rangle=&\langle \xi_0,P^\varepsilon_tf\rangle+\int_0^t\int_{\rit}(L-L^\varepsilon)P_{t-s}^\varepsilon f(x)\xi_s(dx)ds\notag\\ 
 &+\int_0^t \int_{\rit}
\bigg((b(x,V*\xi_s(x))-d(x,U*\xi_s(x)))P^\varepsilon_{t-s}f(x)\notag\\
&+\int_{\rit\setminus(-1,1)}\left(P^\varepsilon_{t-s}f(x+\hat{\sigma}(x)h)-P^\varepsilon_{t-s}f(x)\right)\frac{dh}{|h|^{1+\alpha}}\bigg)\xi_s(dx)ds.
\end{align}
 One now whishes to make $\varepsilon$ tend to $0$.  Since, 
\begin{align*}
  (L-L^\varepsilon)P_{t-s}^\varepsilon f(x)=\int_{(-1,1)}\int_{\hat{\sigma}^\varepsilon(x)h}^{\hat{\sigma}(x)h}\left(\partial_xP_{t-s}^\varepsilon f(x+y)-\partial_xP_{t-s}^\varepsilon f(x)\right)dy\frac{dh}{|h|^{1+\alpha}}, 
\end{align*}
  Lemma \ref{contderpeps}, and the estimation $\|\hat{\sigma}-\hat{\sigma}^\varepsilon\|_\infty\leq C\sqrt{\varepsilon}$, imply
  that
$$|(L-L^\varepsilon)P_{t-s}^\varepsilon f(x)|\leq C\frac{|\hat{\sigma}(x)-\hat{\sigma}^\varepsilon(x)|}{\varepsilon^{\alpha/4}}\int_{(-1,1)}\frac{dh}{|h|^{\alpha/2}}\leq C\varepsilon^{(2-\alpha)/4}.$$

\noindent Letting $\varepsilon\to 0$ in \eqref{readif3bis}, one concludes using \eqref{cvsg} that\begin{align}
\label{readif3} \langle
\xi_t,f\rangle=&\langle \xi_0,P_tf\rangle+ \int_0^t \int_{\rit}
\bigg((b(x,V*\xi_s(x))-d(x,U*\xi_s(x)))P_{t-s}f(x)\notag\\
&+\int_{\rit\setminus(-1,1)}\left(P_{t-s}f(x+\hat{\sigma}(x)h)-P_{t-s}f(x)\right)\frac{dh}{|h|^{1+\alpha}}\bigg)\xi_s(dx)ds.
\end{align}
% \begin{equation}\label{kolmog}
%    \begin{cases}
%       \partial_s \psi(s,x)+L\psi(s,x)=0,\;(s,x)\in[0,t]\times\rit\\
% \psi(t,x)=f(x),\;x\in\rit
%    \end{cases}
% \end{equation}

%\noindent The proof of the following lemma is postponed to the end of the section.

\bigskip

\bigskip
\noindent
Next, we consider the
variation norm defined for $\mu_1, \mu_2\in M_F$ as follows:
\begin{equation}
|| \mu_1 - \mu_2 || = \sup_{f \in
  L^\infty({\rit}), \; | | f ||_\infty \leq 1}
|\left<\mu_1-\mu_2,f\right>|.
\end{equation}
By the Jordan-Hahn decomposition of the measure $\mu_1 - \mu_2$, there exists a Borel subset $A$ of $\rit$ such that $||\mu_1 - \mu_2||=\left<\mu_1 - \mu_2,1_A-1_{\rit\setminus A}\right>$. In view of the  inner regularity of the measure $\mu_1 + \mu_2$ there exists a closed set $B\subset A$ such that $\left<\mu_1 + \mu_2,A\setminus B\right>$ is arbitrarily small. The function $f_k(x)=(1-kd(x,B))\vee(-1)$ is Lipschitz continuous and tends to $1_B-1_{\rit\setminus B}$, as $k$ tends to $\infty$. By Lebesgue's theorem, $\left<\mu_1 - \mu_2,f_k(x)\right>$ tends to $\left<\mu_1 - \mu_2,1_B-1_{\rit\setminus B}\right>$, as $k\to\infty$. Since 
$$
|||\mu_1 - \mu_2||-\left<\mu_1 - \mu_2,1_B-1_{\rit\setminus B}\right>|=2|\left<\mu_1 - \mu_2,A\setminus B\right>|\leq 2\left<\mu_1 + \mu_2,A\setminus B\right>,
$$ for $B$ and $k$ well chosen,  $\left<\mu_1 - \mu_2,f_k\right>$ is arbitrarily close to $||\mu_1 - \mu_2||$. Now, utilizing the convolution, $f_k$ may be approximated by a sequence of $C^2_b$ functions globally bounded by $1$ which converge uniformly on compact sets. Thus one deduces that $|| \mu_1 - \mu_2 || = \sup_{f \in
  C^2_b({\rit}), \; | | f ||_\infty \leq 1}
|\left<\mu_1-\mu_2,f\right>|.$

\medskip
Now, we are ready  to  prove the uniqueness of a solution of~(\ref{readif3}).
For two solutions $(\xi_t)_{t\geq 0}$, and $(\bar
\xi_t)_{t\geq 0}$, of~(\ref{readif3}), such that  $\sup_{t\in [0,T]}
\left< \xi_t+\bar \xi_t,1 \right> =A_T <+\infty$,  and $f\in C^2_b(\rit)$ satisfying the condition  $|| f ||_\infty \leq 1$, one has
\begin{align}
\label{lip}
| \left<\xi_t-\bar \xi_t,f\right> |\leq&
\intot \bigg| \int_{\rit} [\xi_s(dx) - \bar \xi_s(dx)]
  \bigg((b(x,V*\xi_s(x))-d(x,U*\xi_s(x)))P_{t-s}f(x)\notag\\
  &\phantom{\leq
\intot \bigg| \int_{\rit} [\xi_s(dx) - \bar \xi_s(dx)]
  \bigg( }
  \hskip -2cm+\int_{|h|\geq 1}(P_{t-s}f(x+\hat{\sigma}(x)h)-P_{t-s}f(x))\frac{dh}{|h|^{1+\alpha}}\bigg)\bigg| ds
  \notag \\ 
   & +  \intot \left|\int_{\rit} \bar
  \xi_s(dx)(b(x,V*\xi_s(x))-b(x,V*\bar{\xi}_s(x)))\right|
  ds \notag \\ 
   &+   \intot \left|\int_{\rit} \bar
  \xi_s(dx)(d(x,U*\xi_s(x))-d(x,U*\bar{\xi}_s(x)))P_{t-s}f(x)\right|
  ds.
\end{align}
Since $|| f ||_\infty \leq 1$, then $|| P_{t-s}f ||_\infty \leq 1$
and, for all $x\in \rit$,
$$
 \left|(b(x,V*\xi_s(x))-d(x,U*\xi_s(x)))P_{t-s}f(x)\right|\leq
\bar{b}+\bar{d}(1+\bar{U}A_T),
$$
and
$$ \left| \int_{|h|\geq 1}(P_{t-s}f(x+\hat{\sigma}(x)h)-P_{t-s}f(x))\frac{dh}{|h|^{1+\alpha}}\right| \leq \frac{4}{\alpha}
.$$
Moreover, $b$ and $d$ are Lipschitz continuous in their second
variable with respective constants $K_b$, and $K_d$. Thus we obtain
from~(\ref{lip}) that
\begin{equation}
| \left<\xi_t-\bar \xi_t,f\right> | \leq
\left[\bar{b}+\bar{d}(1+\bar{U}A_T)+\frac{4}{\alpha}+K_b A_T\bar{V}+K_d A_T\bar{U}
\right] \intot  || \xi_s - \bar\xi_s ||ds.
\label{majodualvar}
\end{equation}
Taking the supremum over $C^2_b$ functions $f$ bounded by $1$, one obtains
$$ ||\xi_t - \bar\xi_t ||\leq
\left[\bar{b}+\bar{d}(1+\bar{U}A_T)+\frac{4}{\alpha}+K_b A_T\bar{V}+K_d A_T\bar{U}
\right] \intot  || \xi_s - \bar\xi_s ||ds, $$
and uniqueness follows by an application of Gronwall's Lemma.

\bigskip
\noindent
{\bf Step 5} \phantom{9} In the final step we shall prove the existence of a density claimed in Part {\it (iii)}. 
The image of the Poisson random measure $N$ on $\rit_+\times  [-1, 1]$ with intensity $dt  \frac{dh}{ |h|^{1+\alpha}}$,  by the mapping $(t,h)\mapsto (t,\frac{{\rm sgn}(h)}{\alpha |h|^\alpha})$, is a Poisson random measure on $\rit_+\times \rit\setminus [-\frac{1}{\alpha},\frac{1}{\alpha}]$ with intensity $dtdz$ since, for $z=\frac{{\rm sgn}(h)}{\alpha |h|^\alpha}$, one has $dz=\frac{dh}{|h|^{1+\alpha}}$. Let us denote by $\tilde{\mu}(dt,dz)$ the associated compensated measure. The stochastic differential equation \eqref{edspetitsauts} can now be written in the form,
$$
X^x_t=x+\int_{(0,t]\times \rit\setminus [-\frac{1}{\alpha},\frac{1}{\alpha}]} c(X^x_{s^-},z) \tilde{\mu}(ds,dz),
$$ 
for $c(x,z)= \hat{\sigma}(x)\times\frac{{\rm sgn}(z)}{(\alpha |z|)^{1\over \alpha}}$. When the strictly positive function $\hat{\sigma}\in C^3_b$, i.e., it is   bounded together with its derivatives up to order 3,   one may apply Theorem 2.14,  p.11,  \cite{bichtgravjac},  to deduce that, for $t\in (0,T]$, $X^x_t$ admits a density $p_t(x,y)$ with respect to the Lebesgue measure on the real line. With \eqref{readif3}, one deduces that, for $t>0$, $\xi_t$ admits a density equal to 
\begin{align*}
y\mapsto &\int_{\rit}p_t(x,y)\xi_0(dx)+\int_0^t\int_{\rit}(b(x,V*\xi_s(x))-d(x,U*\xi_s(x)))p_{t-s}(x,y)\xi_s(dx)ds
\\&+\int_0^t\int_{\rit}\int_{|h|\geq 1}(p_{t-s}(x+\hat{\sigma}(x)h,y)- p_{t-s}(x,y))\frac{dh}{|h|^{1+\alpha}}\xi_s(dx)ds.
\end{align*}
This completes the proof of Theorem \ref{readif}.
{\hfill $\square$ \vspace{0.25cm}}
% Our aim is to prove that under smooth conditions, and there is a unique solution of this equation and that for each $t>0$, $\xi_t$ has a density with respect to Lebesgue measure. To this aim, we will associate the underlying process. 

% \begin{lem} For $c(x,z) = {\rm sg}(z)\left(\frac{\tilde\sigma(x)}{\alpha |z|}\right)^{1\over \alpha}$ and $\tilde{D}^\alpha f(x)=\int_{|z|>1}(f(x+c(x,z)) - f(x) - c(x,z) f'(x)) dz$, we have
% \be
% \tilde\sigma(x)D^\alpha f(x)= \int_{|z|\leq 1}(f(x+c(x,z)) - f(x)) dz+\tilde{D}^\alpha f(x).
%       \ee
%       \end{lem}

%       \begin{proof} 
% Since $D^\alpha f(x)=\int_0^{+\infty}(f(x+h)+f(x-h)-2f(x))\frac{dh}{h^{1+\alpha}}$, the change of variables $h=\left(\frac{\tilde\sigma(x)}{\alpha z}\right)^{1\over \alpha}$ which is such that
% $\tilde\sigma(x)dh=-\left(\frac{\tilde\sigma(x)}{\alpha z}\right)^{1 +{1\over \alpha}}dz=-h^{1+\alpha}dz$ yields the desired equality.

% % Let us firstly assume

%       $$
% \int_H^{+\infty} {1\over |h|^{1+\alpha} } dh = {1\over \alpha}
% H^{-\alpha}
% $$
%   Montrons que le changement $z\to sg(z) \alpha
% |z|^{-{1\over \alpha}}.$ \begin{eqnarray} \int_\mathbb{R}
% f({sg(z)\over \alpha} |z|^{-\alpha}) dz &=&\int_0^{+\infty}
% f({sg(z)\over \alpha} |z|^{-\alpha}) dz + \int_{-\infty}^0
% f({sg(z)\over \alpha} |z|^{-\alpha}) dz\\
% &=& \alpha^{\alpha+1} \int_0^{+\infty} (f(x) + f(-x)) {dx\over
% x^{1+\alpha}}.
% \end{eqnarray}

% Avec $c(x)$ en plus, on doit prendre ${c(x)\over \alpha}^{1\over
% \alpha} sg(z) |z|^{-{1\over \alpha}}$, qui sera le $c(x,z)$ dans
% les notations de BGJ.

\bigskip
\noindent
Let us now turn to the proof of  Theorem~\ref{readifstoch}.  In this case the allometric exponent $\eta=1$.

\paragraph{Proof of Theorem~\ref{readifstoch}}
We will use a  method similar to  the one employed in the proof of Theorem \ref{readif}.  
Actually, Steps~1, 2 and~3 are completely analogous and we omit them. Thus, we only have to prove the uniqueness (in law)
of the solution of  the martingale
problem~(\ref{eca})--(\ref{qvmf}), and the fact that  any
accumulation point of the sequence of laws of $\nu^K$ is a solution
of ~(\ref{eca})--(\ref{qvmf}). 

\medskip
 
{\it  Uniqueness.}    The uniqueness in the general case can be deduced from the special case  when $b=d=0$ by using the Dawson-Girsanov transform for measure-valued processes  (cf. Evans and
Perkins~\cite{EP94} (Theorem~2.3)).  Indeed,  
\begin{equation*}
\E\left(\intot \intrd
  [b(x,V*X_s(x))-d(x,U*X_s(x))]^2X_s(dx) ds\right)<+\infty,
\end{equation*}
which allows us to use this transform. 

  In the case $b=d=0$ the proof of  uniqueness can be  adapted from Fitzsimmons \cite{F92}, Corollary 2.23. This proof is based on the identification of the Laplace transform of the process, using the extension of the  martingale problem  \eqref{dfmf} to functions  $\psi(s,x)=P_{t-s}f(x)$ with  bounded functions $f$ (Fitzsimmons \cite{F92} Proposition 2.13). For $\psi(s,x)$ being $C^1$ in time, and $C^2$ in $x$,  \eqref{dfmf} extends with an additional term $\partial_{s}\psi$ appearing in the drift part.  As we have seen above, the function $P_{t-s}f(x)$ is not smooth enough. Thus we firstly apply \eqref{dfmf} to $\psi(s,x)=P^\varepsilon_{t-s}f(x)$, as already done in the proof of Theorem \ref{readif}, thanks to Lemma \ref{regukolmog}, and we make $\varepsilon$ tend to $0$ as in   Lemma \ref{contderpeps}. 

\medskip
{\it Identification of the limit.} \phantom{9} Fnally, let us  identify the limit of the sequence of laws of $\nu^K$ as  a solution
of ~(\ref{eca})--(\ref{qvmf}). 
 Write
$Q^K=\loi(\nu^K)$,  and denote by $Q$ a limiting value in ${\cal P}(\mathbb{D}([0,T],(M_F,w))$ of a subsequence (denoted also  $Q^K$), and by $X=(X_t)_{t\geq 0}$ a process with law $Q$.
Because of Step~4, $X$ belongs a.s.~to $C([0,T],(M_F,w))$. We have to
show that $X$ satisfies the conditions~(\ref{eca}), (\ref{dfmf})
and~(\ref{qvmf}). First,  note that~(\ref{eca}) is straightforward
from~(\ref{eq:X4}) and Lemma \ref{mep}. 

Next we show that for any function $f$ in
$C^2_b(\mathbb{R})$, compactly supported when $\alpha\leq 1$, the process $\bar M^f_t$ defined
by~(\ref{dfmf}) is a martingale (the extension to every function
in $C^2_b$ is not hard). Indeed,  consider $0\leq s_1\leq...\leq
s_n<s<t$, and  continuous bounded maps $\phi_1,...\phi_n$ on
$M_F$. Our goal is to prove that, if the function $\Psi$ from
$\mathbb{D}([0,T],M_F)$ into $\mathbb{R}$ is defined by the expression
\begin{align}
&\Psi(\nu) = \prod_{k=1}^n\phi_k(\nu_{s_k})  \Big\{
\langle \nu_t,f\rangle -\langle \nu_s,f\rangle
\notag \\ & -\! \int_s^t\! \intrd\! \bigg(\tilde{\sigma}(x)D^\alpha f(x)+ f(x)
\left[b(x,V*\nu_u(x))-d(x,U*\nu_u(x)) \right]\bigg)\nu_u(dx)du
\Big\},
\end{align}
then
\begin{equation}
\label{cqfd44}
\E\left( \Psi(X) \right)=0.
\end{equation}
It follows from~(\ref{mart1}) that
\begin{eqnarray}
\label{262626}
0=\E \left( \prod_{k=1}^n\phi_k(\nu^K_{s_k})\left\{ M^{K,f}_t -
  M^{K,f}_s \right\}\right)=\E \left( \Psi(\nu^K) \right) - A_K,
\end{eqnarray}
where 
\begin{multline*}
A_K=\E \Big(\prod_{k=1}^n\phi_k(\nu^K_{s_k})\int_s^t \intrd
  \Big\{ p(x) b(x,V*\nu^K_u(x))\Big[\intrd  (f(x+h)- f(x))M_K(x,dh)\Big]\\
  +p(x) r(x) K \int_{\rit}  (f(x+h) - f(x))M_K(x,dh)-\tilde{\sigma}(x)D^\alpha f(x)
\Big\}\nu^K_u(dx)du\Big).
\end{multline*}
In view of  Proposition \ref{restlapfrac},  $A_K$ tends to zero,  as $K$
grows to infinity. Applying   Lemma \ref{mep}, for $p=3$, we see that  the sequence
$(|\Psi(\nu^K)|)_K$ is uniformly integrable, so that
\begin{equation}
\label{333}
\lim_K \E\left(|\Psi (\nu^K)|\right) = \E_Q\left(|\Psi(X)| \right),
\end{equation}
since the function $\psi$ is continuous a.s. at $X$. 
Collecting the previous results allows us to conclude
that~(\ref{cqfd44}) holds true, and thus $\bar M^f$ is a martingale.

Finally, we have to show that the bracket of $\bar M^f$ is of the form
$$
\langle \bar M^f\rangle_t=2\int_0^t \int_{\rit}r(x) f^2(x) X_s(dx)ds.
$$ 
To this end, we first check that
\begin{align}
\label{lfalqoc}
\bar N^{f}_t &= \langle X_t,f \rangle^2 - \langle X_0,f \rangle^2
-\intot  \int_{{\mathbb{R}}} 2r(x)f^2(x) X_s(dx)ds \notag \\
&- 2\intot \langle X_s,f \rangle  \int_{\rit}  f(x)
\left[b(x,V*X_s(x))-d(x,U*X_s(x)) \right]X_s(dx) ds \notag \\
&- 2\intot   \langle X_s,f \rangle  \int_{\rit}
\tilde{\sigma}(x)D^\alpha f(x) X_s(dx) ds
\end{align}
is a martingale. This can be done exactly as for $\bar M^f_t$,
using the semimartingale
decomposition of $\langle \nu^K_t,f \rangle^2$, given by
(\ref{eq:mart-gal}) with $\phi(\nu)=\langle\nu,f\rangle^2$,  and applying Lemma \ref{mep} with $p=4$.
On the other hand, It\^o's formula implies that
\begin{multline*}
\langle X_t,f \rangle^2 - \langle X_0,f \rangle^2 - \langle \bar
M^f\rangle_t
- 2\intot  \langle X_s,f \rangle \int_{\rit}
\tilde{\sigma}(x)D^\alpha f(x) X_s(dx)ds \\
- 2\intot  \langle X_s,f \rangle  \int_{\rit}  f(x)
\big[b(x,V*X_s(x))-d(x,U*X_s(x)) \big]X_s(dx) ds
\end{multline*}
is a martingale. Comparing this formula with~(\ref{lfalqoc}), we
obtain (\ref{qvmf}).\hfill$\Box$
\bigskip

%%%%%%%%%%%%%%%%%%%%%%%%%%%%%%%%%%%%%%%%%%%%%%%%%%%%%%%%%%%%%%%%

\section{Concluding remarks} 

We have developed  models for population dynamics in the context of evolutionary ecology permitting heavy tailed distribution of mutations. Depending on the value of the allometric exponent $\eta$, the continuum (macro) limits of the individual (micro) dynamics turned out to be described by deterministic solutions of fractional nonlocal  reaction-diffusion equations driven by fractional Laplacians (the case $0<\eta<1$), or measure-valued nonlinear stochastic superprocesses driven by L\'evy- stable processes.
These limiting models can now be used as approximate objects  for numerical simulation of  evolutionary Darwinian dynamics in presence of  non-negligible large mutations. Of course, estimators of the relevant parameters of the phenomena under study have to be obtained first.

It the future it  may also be worthwhile, from the perspective of practical applications,  to elucidate  the situation where the mutations   have distributions    intermediate between the heavy tailed distributions studied in this paper and the Gaussian distributions considered in \cite{CFM06}. Such distributions,  which can display a  multiscaling behavior,  $\alpha$-stable type  for small mutations, and exponential for large mutations, have been recently suggested in the physical and economics  literature and studied under different names  such as  truncated L\'evy, and tempered L\'evy distributions, see, e.g.,  \cite{ms94},  \cite{cgmy02}, \cite{r07},  \cite{tw06},   and \cite{cr07}.

%%%%%%%%%%%%%%%%%%%%%%%%%%%%%%%%%%%%%%%%%%%%%%%%%%%%%%%%%%%%%%%%

\end{document}